\documentclass[11pt]{amsart}
\usepackage[latin1]{inputenc}
\usepackage{amssymb}
\usepackage{pdfsync}
\usepackage[english]{babel}

\usepackage[pdftex, colorlinks=true,urlcolor=blue,linkcolor=blue,citecolor=blue]{hyperref}
\usepackage{graphicx}
\graphicspath{ {./images/} }
\usepackage[capitalize]{cleveref}
\usepackage{fullpage}
\usepackage{color}
\DeclareMathOperator{\Ima}{Im}
\usepackage{amsmath}
\usepackage{amsfonts}
\usepackage{mathrsfs}
\usepackage{t1enc , graphicx}
\usepackage{verbatim}
\usepackage{bbm}
\usepackage{enumitem}

\usepackage[left= 1.3 in, right= 1.3 in ,top= 1 in, bottom = 2.1 in]{geometry}

\newcommand{\Z}{\mathbb{Z}}
\newcommand{\N}{\mathbb{N}}
\newcommand{\F}{\mathbb{F}}

\newcommand{\ifff}{if and only if }

\newcommand\Fq{{\mathbb{F}_q}}

\newtheorem{theorem}{Theorem}
\newtheorem*{theorem*}{Theorem}
\newtheorem{proposition}[theorem]{Proposition}
\newtheorem{lemma}[theorem]{Lemma}

\newtheorem*{corollary*}{Corollary}

\theoremstyle{definition}
\newtheorem*{definition*}{Definition}
\newtheorem{definition}[theorem]{Definition}

\theoremstyle{remark}

\newtheorem*{remark*}{Remark}

\newcommand{\vertiii}[1]{{\left\vert\kern-0.25ex\left\vert\kern-0.25ex\left\vert #1 
		\right\vert\kern-0.25ex\right\vert\kern-0.25ex\right\vert}}

\begin{document}
    \author{Pierre-Yves Bienvenu}
    \address{Johann Radon Institute for Computational and Applied Mathematics, Austrian 
Academy of Sciences, Linz, Austria}
\email{pierre.bienvenu@oeaw.ac.at}

		\author{Th\'ai Ho\`ang L\^e}
	
		\address{Department of Mathematics\\
				University of Mississippi\\
				University, MS 38677, USA}
    \email{leth@olemiss.edu} 
		
		        \author{Gauree Wathodkar}
				\address{Department of Mathematics and Statistics\\
				Loyola University Chicago\\
				1032 W. Sheridan Road, Chicago, IL 60660, USA}
		\email{gwathodkar@luc.edu}

\title{A Generalization of S\'ark\"ozy's theorem in function fields}
\begin{abstract}
S\'ark\"ozy's theorem says that if $A \subset \mathbb{Z}$ has positive upper asymptotic density, then there are distinct $a_1, a_2 \in A$ and $n \in \Z$ such that $a_1-a_2 = n^2$. The same is true if $n^2$ is replaced by $F(n)$ for any polynomial $F \in \Z[x]$ with constant term zero. Green proved an $\mathbb{F}_q[t]$-analog of S\'ark\"ozy's theorem with strong quantitative bounds, but required a technical condition on the number of roots of the polynomial $F \in \Fq[x]$. This condition was recently removed by Li and Sauermann. In this paper, we generalize Green's argument to accommodate equations in more variables in $\Fq[t]$, while pointing out that the technical condition can be removed by means of a simple observation.
\end{abstract}	
\maketitle
\section{Introduction}
The \textit{upper asymptotic density} of a set $A \subset \mathbb{Z}$ is given by 
\begin{equation*}
    \overline{d}(A)=\limsup_{N\rightarrow \infty} \frac{|A\cap \{1,\dots,N\}|}{N}.
\end{equation*}
S\'ark\"ozy \cite{sar} proved that if $\overline{d}(A) > 0$, then there are distinct $a_1, a_2 \in A$ and $n \in \Z$ such that $a_1-a_2 =n^2$. The same conclusion holds when $n^2$ is replaced by $F(n)$, where $F \in \Z[x]$ is a polynomial with constant term zero. More generally, the conclusion holds if and only if $F$ is an \textit{intersective polynomial}, that is, for any $d \in \Z^+$, there exists $n \in \Z$ such that $d| F(n)$. This condition is easily seen to be necessary, by taking $A = d\Z$. Sufficiency follows from work of Kamae and Mend\`es France \cite{km}. We refer the reader to the survey \cite{le-survey} for a discussion on sets enjoying the same property as $\{ n^2 \}$, known as \textit{intersective sets}. 

The finitary version of (the generalization of) S\'ark\"ozy's theorem is as follows:

\begin{theorem}[Generalized S\'ark\"ozy] \label{sarkfi}
    Let $F \in \Z[x]$ be intersective. For $N\in \N$, let $\alpha_F(N)$ be the maximum size of a subset $A \subset \{1,\dots,N\}$ that do not contain distinct elements $a_1, a_2$ such that $a_1 - a_2 = F(n)$ for some $n \in \Z$. Then $\alpha_F(N) = o(N)$ as $N\rightarrow \infty$.     
\end{theorem}

To date, the best bound is $\alpha_F(N) \leq C N e^{- c (\log N)^d } $, for any $0 < d < 1/4$, where $C$ and $c$ are positive constants depending on $F$ and $d$. It is due to Adajar et. al. \cite{adajar}, building on the breakthrough of Green and Sawhney \cite{green-sawhney} for the case $F(n)=n^2$. 

Let $\Fq$ be a finite field with $q$ elements, where $q$ is a power of a prime $p$. Let $\Fq[t]$ be the ring of polynomials with coefficients in $\Fq$ in an indeterminate $t$. Also, let $G_n \subset \Fq[t]$ be the set of all polynomials in $\Fq[t]$ with degree less than $n$. In \cite{le}, Liu and the second author of the present paper studied the analogue of 
S\'ark\"ozy's theorem in $\Fq[t]$ and proved the following, using the circle method.

\begin{theorem}[L\^e-Liu] \label{leliu}
     Let $A\subset G_n$ with $|A|\gg q^n (\log n)^7/n$. Then, there are distinct polynomials $a_1(t),a_2(t) \in A$ such that $a_1(t)-a_2(t)=(b(t))^2$ for some $b(t)\in\mathbb{F}_q[t]$.
\end{theorem}
 
Using the polynomial method developed by Croot-Lev-Pach \cite{clp}, Green \cite{green} gave a much stronger bound than Theorem \ref{leliu}.
\begin{theorem}[Green] \label{gr}
	Let $F(x) = \alpha_1 x + \alpha_2 x^2 + \cdots + \alpha_\ell\ x^\ell \in\Fq[x]$ be a polynomial of degree $\ell$ with constant term zero such that the number of roots of $F$ in $\Fq[t]$ is coprime to $q$. Then there is a constant $c=c(q,\ell) \in (0,1)$ such that if $A\subset G_n$ has cardinality greater than $2q^{cn}$, then there are distinct polynomials $a_1(t), a_2(t)\in A$ such that $a_1(t)-a_2(t) = F(b(t))$ for some $b(t)\in \Fq[t]$.
\end{theorem}

In particular, Theorem \ref{gr} applies when $F(x)=x^\ell$, but does not apply when $F(x) = x^{\ell-p+1}(x-1)\cdots(x-p+1)$. The condition on the number of roots of $F$ is not quite satisfactory, and is conceivably an artifact of the proof. This condition is not needed qualitatively (see \cite[Section 6]{llw}).
Recently, Li and Sauermann \cite{li} showed that the root condition on $F$ can indeed be removed.

\begin{theorem}[Li-Sauermann] \label{th:ls}
Theorem \ref{gr} remains valid without the condition on the number of roots of $F$, at the cost of adjusting $2 q^{cn}$ to $c' q^{cn}$, for some constant $c' = c'(q, \ell)>0$.
\end{theorem}

Li and Sauermann proved Theorem \ref{th:ls} by generalizing a key lemma of Green (Lemma \ref{greenlemma} below). Alternatively, Theorem \ref{th:ls} follows from the following  observation. Green only considered $F \in \Fq[x]$, but an inspection shows that the same proof works for $F\in \Fq[t][x]$, at the cost of adjusting $2q^{cn}$ to $c'q^{cn}$. (Note that when $F \in \Fq[x]$, any root in $\Fq[t]$ of $F$ must necessarily be in $\Fq$.)
Now, suppose 
\[
F(x) = \alpha_1(t) x + \alpha_2(t) x^2 + \cdots + \alpha_\ell(t) x^\ell \in \Fq[t][x]
\]
is an arbitrary polynomial with constant term zero (and no assumption on the number of roots). Then 
\[
F(t^m x) = \alpha_1(t) t^m x + \alpha_2(t) t^{2m} x^2 + \cdots + \alpha_\ell(t) t^{\ell m} x^\ell.
\]
If $m$ is greater than $\deg_t \alpha_i(t)$ for all $i$, then the degrees of $t$ in the non-zero terms $\alpha_i(t) t^{im} x^i$ are all distinct. Therefore, the only root in $\Fq[t]$ of $F(t^m x)$ is $x = 0$. Applying Theorem \ref{gr} with $F(t^m x)$ in place of $F$, we see that any set $A \subset G_n$ with cardinality greater than $c' q^{cn}$, must contain distinct elements $a_1(t), a_2(t)$ such that $a_1(t)-a_2(t) = F(t^m b(t))$ for some $b(t)\in \Fq[t]$. Hence, Theorem \ref{th:ls} follows.

For example, if $F(x) = x^{\ell-p+1}(x-1)\cdots(x-p+1)$, then $F(tx) = t^{\ell-p+1} x^{\ell-p+1}(tx-1)\cdots(tx-p+1)$ has only one root $x=0$ in $\Fq[t]$.

In this paper, we generalize Theorems \ref{gr} and \ref{th:ls} to equations in more variables.

\begin{theorem}\label{thmA}
	Let $r, k, d$ and $\ell$ be integers satisfying $k>2r^2 (1-1/\ell^2)$. Suppose $c_1,\dots , c_k\in \Fq[t]$ have degree at most $d$ and satisfy $\sum_{i=1}^{k}c_i=0$. Let $F(x) \in \Fq[t][x]$ be a polynomial of degree $\ell$ with constant term zero. 
 Then there exist constants $c = c(q,r, k ,\ell) \in (0,1)$ and $c'=c'(q,r,k,d,F)>0$ such that the following holds: Any $A\subset G_n$ satisfying $|A|> c' \cdot q^{cn} $ must contain a nontrivial solution to 
	\begin{equation}\label{ee}
		c_1 a_1^r+\dots +c_k a_k^r=F(b)
	\end{equation}
	for some $b\in \Fq[t]$.
\end{theorem}
By a trivial solution we mean $(a_1,\dots,a_k)=(a,\dots,a)$ for some $a\in A$ and $b=0$. Theorems \ref{gr} and \ref{th:ls} are special cases of Theorem \ref{thmA} when $r=1$ and $k=2$.

Theorem \ref{thmA} was inspired by the following result \cite{bienvenu} of the first author of the present paper.

\begin{theorem}[Bienvenu] \label{bi}
	Let, $r, k$ and $d$ be integers satisfying $k\geq 2r^2+1$. Suppose $c_1, \dots, c_k\in \Fq[t]$ have degree at most $d$ and $\sum_{i=1}^k c_i=0$. Then there exist constants $c(r,q)\in(0,1)$ and $c'= c'(q,r,k,d) >0$ such that the following holds: Any $A\subset G_n$ with $|A|\geq c' \cdot q^{cn}$ must contain a non-trivial solution to 
	\begin{equation}\label{ee2}
		c_1 a_1^r+\dots +c_k a_k^r=0.
	\end{equation}
\end{theorem}

Again, by a trivial solution we mean $(a_1,\dots,a_k)=(a,\dots,a)$ for some $a\in A$. 

Note that Theorem \ref{bi} provides a stronger conclusion than Theorem \ref{thmA}, but the number of variables required in Theorem \ref{bi} is greater than the number of variables required in Theorem \ref{thmA}. Therefore, Theorems \ref{thmA} and \ref{bi} do not subsume each other.

The organization of the paper is as follows. In Section \ref{sec:prop}, we deduce Theorem \ref{thmA} from a general statement (Proposition \ref{thmB}). In Section \ref{sec:th}, we prove Proposition \ref{thmB} using the Croot-Lev-Pach polynomial method.

\section{Proof of Theorem \ref{thmA}} \label{sec:prop}

Theorem \ref{thmA} is a consequence of the following proposition, which is a hybrid of \cite[Proposition 2]{bienvenu} and \cite[Theorem 1.3]{green}.
We say a map $\Psi: \mathbb{F}_q^n \rightarrow \mathbb{F}_q^m$ is a polynomial map of degree at most $d$ if $\Psi(x_1, x_2, \ldots, x_{n}) = (\psi_1(x_1, x_2, \ldots, x_{n}), \ldots, \psi_{m}(x_1, x_2, \ldots, x_{n}))$, where each $\psi_i(x_1, \ldots, x_n)$ is a polynomial of degree at most $d$.

\begin{proposition}\label{thmB}
Let $m,n,n',d'$ and $d''$ be positive integers. Let $\Psi: (\mathbb{F}_q^n)^k \rightarrow \mathbb{F}_q^{n'}$ be a polynomial map with degree at most $d'$ and $\Phi:\mathbb{F}_q^m \rightarrow \mathbb{F}_q^{n'}$ be a polynomial map with degree at most $d''$.
Let $A \subset \F_q^n$. Assume that
    \begin{enumerate}[label=(\roman*)]
        \item $\Psi^{-1}(0)\cap A^k = \{ (a, ..., a) : a \in A \}$,
        \item $|\Phi^{-1}(0)|$ is coprime to $q$, and
        \item $\Psi(A^k) \cap \Ima(\Phi) = \{0\}$.
    \end{enumerate}
    Then,
	\begin{equation} \label{eq:prop}
		|A|\leq k \cdot \inf_{0<x<1}\frac{(1+x+\dots+x^{q-1})^n}{x^{\left(\frac{q-1}{k}\right)\left(n'-\frac{m}{d''}\right)d' }}.
	\end{equation}
\end{proposition}

\begin{proof}[Proof of Theorem \ref{thmA} assuming Proposition \ref{thmB}]
By the observation made in the introduction, we can assume that $F(x)$ has only one root $x=0$ in $\Fq[t]$, by replacing $F(x)$ with $F(t^Mx)$ for some sufficiently large $M$.

We identify $G_n$ with $\F_q^n$, by identifying each polynomial $g=g_0+g_1t+\dots+ g_{n-1}t^{n-1} \in G_n$ with the vector $(g_0,g_1, \dots , g_{n-1})\in \mathbb{F}_q^n$. Then $g^r \in G_{(n-1)r+1}$ corresponds to the vector $(g_0^r, rg_0^{r-1}g_1,\ldots,g_{n-1}^r) \in \mathbb{F}_q^{r(n-1)+1}$. Note that each component of $g^r$ is a polynomial of degree at most $r$ in $g_0,g_1, \dots , g_{n-1}$. 
	
Under this identification, the map 
$
	(a_1,\dots, a_k) \mapsto \sum ^k_{i=1} c_i a_i^r
$
induces a polynomial map 
\begin{align*}
\Psi: (\mathbb{F}_q^n)^k & \rightarrow \mathbb{F}_q^{n'} \\  
	(a_1,\dots, a_k) & \mapsto (\psi_1 (a_1,\dots, a_k), \ldots, \psi_{n'} (a_1,\dots, a_k) )
\end{align*}	
where $n':=(n-1)r+d+1$ and $\psi_i (a_1,\dots, a_k)$ is the coefficient of $t^{i-1}$ in $\sum ^k_{i=1} c_i a_i^r$. Moreover, the degree of each $\psi_i$ is at most $d':=r$.
	
Suppose 
\[
F(x) = \alpha_1(t) x+ \alpha_2(t) x^2+\cdots +\alpha_{\ell}(t) x^\ell \in \Fq[t][x].
\]
Let $u$ be the maximum of $\deg_t \alpha_{i}(t)$ for $i=1, \ldots, \ell$ and $m = \lfloor {\frac{n'-u}{\ell}}\rfloor +1$. If $h= h_0+h_1t+\cdots+h_{m-1}t^{m-1} \in G_m$, then 
	\begin{align*}
		F(h_0+h_1t&+\dots+h_{m-1}t^{m-1})\\
		&= f_0(h_0,\dots, h_{m-1})+f_1(h_0,\dots, h_{m-1})t+\dots+f_{n'-1}(h_0,\dots, h_{m-1})t^{n'-1} \in G_{n'}
	\end{align*}
	where $f_i \in \Fq[x_1, \dots,x_m]$ are polynomials with degrees at most $\ell$ and $f_i(0)=0$, for $0 \leq i\leq n'-1$.
	
Consider the polynomial map 
\begin{align*}
\Phi:\mathbb{F}_q^m & \rightarrow \mathbb{F}_q^{n'}, \\
(h_0, \ldots, h_{m-1}) &\mapsto (f_0(h_0,\dots, h_{m-1}), \ldots, f_{n'-1}(h_0,\dots, h_{m-1})). 
\end{align*}
	
Then $\Phi$ has degree at most $d'':=\ell$.  Suppose now $A \subset G_n$ is a subset that does not contain a nontrivial solution to $c_1 a_1 ^r + \cdots + c_k a_k^r = F(b)$.
We apply Proposition \ref{thmB} to $\Psi, \Phi$ and the set $A$. The first and third condition are satisfied by the assumption on $A$. The second condition is satisfied by the assumption on the number of roots in $\Fq[t]$ of $F$. 

    
Substituting $n'= (n-1)r+d+1, d'=r, d''=\ell$ and $m = \lfloor {\frac{n'-u}{\ell}}\rfloor +1$, \eqref{eq:prop} implies that
\[
|A|\leq k \inf_{0<x<1}\frac{(1+x+\dots+x^{q-1})^n}{ x^{\frac{q-1}{k} \left( 1 - \frac{1}{\ell^2} \right) nr^2 + O_{q,r,k,d,F}(1) }}.
\]
Writing $E := (q-1)r^2(1-1/\ell^2)k^{-1}$, then we have $E <(q-1)/2$ since $k>2r^2(1-1/\ell^2)$.

Let $f(x) =(1+x+\dots+x^{q-1})x^{-E}$. Then $f'(1) = \sum_{i=0}^{q-1} (i-E) > 0$, so $f$ is increasing on some interval $(1-\epsilon,1)$. Therefore, $\inf_{0<x<1} f(x) < f(1) = q$. Let $x_0 \in (0,1)$ be such that $f(x_0) <q$, then
  \begin{equation*}
|A|  \leq k \cdot x_0^{O_{q,r,k,d,F}(1)} \cdot (f(x_0))^n \leq c' \cdot q^{cn} 
  \end{equation*}
for some constants $c = c(q,r,k,\ell) \in (0,1)$ and $c'(q,r,k,d,F) >0$, as desired.
\end{proof}

\section{Proof of Proposition \ref{thmB}} \label{sec:th}
The remainder of the paper is devoted to the proof of Proposition \ref{thmB}. First we recall some preliminaries. The first one is the notion of slice rank introduced by Tao \cite{tao} (see also \cite[Definition 6.5.2]{zhao}).
\begin{definition}[Slice rank of a function]\label{srank}
	Let $A$ be a finite set and $\F$ be a field. A function $f: A^k \rightarrow \F$ is said to have \textit{slice rank} $1$ if it can be written as  
	\begin{equation*}
		f(x_1,x_2, \dots,x_k )= g(x_j)h(x_1,\dots, x_{j-1},x_{j+1},\dots, x_k),
	\end{equation*}
	for some $j\in \{1, \ldots, k\}$ and some nonzero functions $g:A\rightarrow \F$ and $h:A^{k-1}\rightarrow\F$.
	
	The \textit{slice rank} of a function $F: A^k \rightarrow \F$ is the minimum integer $r$ such that $F$ can be written as a sum of $r$ slice rank $1$ functions.
\end{definition}

	

Tao \cite[Lemma 1]{tao} (see also \cite[Lemma 6.5.5]{zhao}) proved the following fact regarding the slice rank of a diagonal function.
\begin{lemma}[Tao] \label{srdiag}
Let $A$ be a finite set.
	Suppose a function $f:A^k\rightarrow \F$ satisfies $f(a_1,\dots, a_k)\neq 0$ \ifff $a_1=a_2=\dots=a_k$. Then $f$ has slice rank $|A|$. 
\end{lemma}

Our next tool is the following key lemma by Green \cite[Lemma 3.1]{green}, which says that the image of a polynomial map can be realized as the support of a low degree polynomial.
\begin{lemma}[Green]\label{greenlemma}
	Let $\Phi:\mathbb{F}_q^m \rightarrow \mathbb{F}_q^{n}$ be a polynomial map with degree at most $d$ such that $|\Phi^{-1}(0)|$ is coprime to $q$. Then there exists a polynomial $P\in \Fq[x_1,x_2,\dots,x_{n}]$ whose degree in every single variable is at most $(q-1)$ such that 
	\begin{enumerate}[label=(\roman*)]
		\item  $\deg(P) \leq (q-1) \left(n-\frac{m}{d}\right)$,
		\item $P$ is supported on $\Ima(\Phi)$ (i.e., $P(x)=0$ if $x\notin \Ima(\Phi)$), and
		\item $P(0)\neq 0$.
	\end{enumerate}
	
\end{lemma}

We are now ready to prove Proposition \ref{thmB}.

\begin{proof}[Proof of Proposition \ref{thmB}]
	
Applying Lemma \ref{greenlemma} to the polynomial map $\Phi: \mathbb{F}_q^m \rightarrow \mathbb{F}_q^{n'}$, we find that there exists a polynomial $P\in \Fq[x_1,x_2,\dots,x_{n'}]$ whose degree in every single variable is at most $(q-1)$ such that 
	\begin{enumerate}[label=(\roman*)]
		\item  $\deg(P) \leq (q-1) \left(n'-\frac{m}{d''}\right)$,
		\item $P$ is supported on $\Ima(\Phi)$, and
		\item $P(0)\neq 0$.
	\end{enumerate}
We now consider the function	$f(a_1, \ldots , a_k) := P(\Psi(a_1, \ldots, a_k))$ from $A^k$ to $\Fq$.
	
\noindent	\textbf{Claim:} $f(a_1, \ldots , a_k) \neq 0 $ if and only if $a_1= \cdots = a_k = a$ for some $a \in A$.

Indeed, for any $a \in A $, $f(a, \ldots , a)= P(\Psi(a, \ldots, a)) = P(0) \neq 0$. Conversely, suppose the $a_i$'s are not all equal. By the assumption on $\Psi$, we have $\Psi(a_1, \ldots, a_k) \neq 0$. Since $\Psi(A^k)$ and $\Ima(\Phi)$ only intersect at $0$, we have that $\Psi(a_1, \ldots, a_k) \notin \Ima(\Phi)$. Therefore, $P(\Psi(a_1, \ldots, a_k)) = 0$ since $P$ is supported on $\Ima(\Phi)$, and the claim is proved.
	
By Lemma \ref{srdiag}, the slice rank of $f$ is $|A|$. We will now estimate the slice rank of $f$ in a different way. Note that $f$ is a polynomial of degree at most $\text{deg}P\cdot\text{deg}\Psi \leq (q-1)\left(n'-\frac{m}{d''}\right)d'=: D$. For each $i=1, \ldots, k$, let $a_i = (a_{i,1}, \ldots, a_{i,n})$. Then $f(a_1, \ldots, a_k)$ can be written as a sum of monomials of the form
	\begin{equation} \label{eq:slice}
		a_{1,1}^{\alpha_{1,1}}a_{1,2}^{\alpha_{1,2}}\cdots a_{1,n}^{\alpha_{1,n}}a_{2,1}^{\alpha_{2,1}} \cdots a_{2,n}^{\alpha_{2,n}}\cdots a_{k,1}^{\alpha_{k,1}} \cdots a_{k,n}^{\alpha_{k,n}}
	\end{equation}
	where $\sum_{i=1}^{k}\sum_{j=1}^{n}\alpha_{i,j}\leq  D$. Since $a^q=a$ for every $a \in \F_q$, we may further assume that $\alpha_{i,j}\leq q-1$ for every $1 \leq i \leq k$ and $1 \leq j \leq n$.
	
	By the pigeonhole principle, there exists $i_0 \in \{ 1, \ldots, k \}$ such that $\sum_{j=1}^{n}\alpha_{i_0,j}\leq  D/k$. 
	For fixed $i_0$ and $(\alpha_{1},\alpha_{2},\dots, \alpha_{n})$, the sum of all terms in $f$ of the form \eqref{eq:slice} with $(\alpha_{i_0, 1},\alpha_{i_0, 2},\dots, \alpha_{i_0, n})= (\alpha_{1},\alpha_{2},\dots, \alpha_{n})$, is a slice rank 1 function. Hence, the slice rank of $f$ is at most $kN$, where $N$ is the
number of tuples $(\alpha_{1},\alpha_{2},\dots, \alpha_{n})$ such that $\alpha_{j}\leq q-1  $ for all $j \in \{1, \ldots, n\}$ and $\sum_{j=1}^{n}\alpha_{j}\leq  D/k$. It remains to estimate $N$. We have
	\begin{equation*}
		N= \sum_{\substack{u_0,u_1,\dots,u_{q-1}\geq 0\\u_0+u_1+\dots+u_{q-1}=n \\ u_1+2u_2+\dots+ (q-1) u_{q-1} \leq D/k } } \frac{n!}{u_0!u_1! \dots u_{q-1}!}.
	\end{equation*}
	 
Note that	 $N$ equals the sum of all the coefficients of $x^{v}$ with $v \leq D/k$ in the expansion of $(1+x+\dots +x^{q-1})^n$. Let $x \in (0,1)$ be arbitrary. When $v \leq D/k$, we have $x^v \geq x^{D/k}$, so
\begin{equation*}
	 	N \cdot x^{D/k} \leq (1+x+\dots+x^{q-1})^n.
\end{equation*}
	 
	 Therefore, the slice rank of $f$ is
	 \begin{equation*}
	 	|A| \leq k N\leq k\cdot \inf_{0<x<1} \frac{(1+x+\dots+x^{q-1})^n}{x^{\left(\frac{q-1}{k}\right)\left(n'-\frac{m}{d''}\right)d'}},
	 \end{equation*}
as desired.	 
\end{proof}

\noindent \textbf{Acknowledgments.} 
The first author
was supported by the Austrian Science Fund (FWF) [10.55776/PAT4719224].
The second and third authors were supported by National Science Foundation Grant DMS-2246921. The second author is also supported by a travel gift from the Simons Foundation. We would like to thank the referees for suggestions which improved the presentation of the paper. This work was created entirely by the authors without the assistance of artificial intelligence tools.

\end{document}